\documentclass[a4,12pt]{amsart}       % onecolumn (second format)

\usepackage[utf8]{inputenc}
\usepackage[T1]{fontenc}
\usepackage{yfonts}
\usepackage[english]{babel}
\usepackage{lipsum}
\usepackage{amsmath}
\usepackage{amsthm}
\usepackage{amssymb}
\usepackage{soul}

\usepackage[shortlabels]{enumitem}
\usepackage{graphicx}
\usepackage{mathtools}
\usepackage{hyperref}
\usepackage{amsfonts}
\usepackage{latexsym}
\usepackage{amscd}
\usepackage{cancel}
\usepackage{cleveref}
\usepackage[all]{xy}
\usepackage[dvipsnames]{xcolor}
\usepackage{stmaryrd}
\def\R#1{\textcolor{Black}{#1}}

\DeclareMathOperator{\Hom}{Hom}
\DeclareMathOperator{\Ext}{Ext}

\DeclareMathOperator{\Path}{Path}

\DeclareMathOperator{\ot}{\overset\to\otimes}

\def\Path{\text{Path}}

\def\x{\overline x}

\def\dualita#1#2{\mathrel{
                 \mathop{\vcenter{
                 \offinterlineskip
                 \hbox to 1.2truecm{$\mapsto$}%\kern ...
                 \hbox to 1.2truecm{$\mapsfrom$}}}%
                 }}
\newtheorem{theorem}{Theorem}[section]
\newtheorem{lemma}[theorem]{Lemma}
\newtheorem{proposition}[theorem]{Proposition}

\newtheorem{definition}[theorem]{Definition}
\newtheorem{remark}[theorem]{Remark}
\newtheorem{example}[theorem]{Example}
\begin{document}
\title[Simple modules over Leavitt path algebras]{Homological properties of simple modules over Leavitt path algebras}
\author{Francesca Mantese}
\address{Dipartimento di Informatica, Universit\`{a} degli Studi di Verona, I-37134 Verona, Italy}
\email{francesca.mantese@univr.it}
\author{Alberto Tonolo}
\address{Dipartimento di Matematica ``Tullio Levi-Civita'', Universit\`{a} degli Studi di Padova, I-35121, Padova, Italy}
\email{alberto.tonolo@unipd.it}

\thanks{
%2010 AMS Subject Classification:  16S99 (primary)     \\ 
%  $ {}^*$corresponding author  \ \ abrams@math.uccs.edu \ \  \\  
The two authors are supported by Project funded by the European Union – NextGenerationEU
under the National Recovery and Resilience Plan (NRRP), Mission 4 Component 2 Investment 1.1 -
Call PRIN 2022 No. 104 of February 2, 2022 of Italian Ministry of University and Research; Project
2022S97PMY (subject area: PE - Physical Sciences and Engineering) “Structures for Quivers, Algebras
and Representations (SQUARE)”. They are moreover members of INDAM - GNSAGA}

\maketitle

\begin{abstract}
Let $K$ be any field, and let $E$ be any graph.  We explicitly construct the projective resolution of simple left modules over the Leavitt path algebra  $L_K(E)$ associated to cycles and irreducible polynomials.  Then we study the dimension of the $K$-vector space of the extensions between two such simple modules.

\footnotesize
Keywords and phrases: Leavitt path algebra; simple modules; irreducible polynomials.

MSC 2020 Subject Classifications:    Primary  16S88, Secondary 16S99
\normalsize

\end{abstract}

\begin{center}
\emph{to Gene: an amazing, funny mathematician and friend}
\end{center}

\section*{Introduction}

The study of the module theory over Leavitt path algebras remains an interesting area of research, with several aspects still requiring clarification. In 2012  Chen \cite{Ch15} introduced a method for constructing simple modules $V_{[p]}$ over a Leavitt path algebra $L_K(E)$ of an arbitrary graph $E$, where $[p]$ denotes the class of infinite paths tail-equivalent to a fixed infinite path $p$ in $E$. Two infinite paths in $E$ are tail equivalent if they are obtained by multiplying the same infinite path on the left  by possibly different finite paths.
In particular, given any cycle $e$ in $E$, one can construct the simple module $V_{[e^\infty]}$ associated with the infinite path obtained by traversing the cycle $e$ infinitely many times.

 Subsequently, Ara and Rangaswamy \cite{AR14}, introduced a new family of simple modules $V^p_{[e^\infty]}$, which generalize Chen's construction. These modules are associated with an exclusive cycle $e$ and an irreducible polynomial $p(x)\in K[x]$ satisfying $p(0)\not=0$.  More recently, giving  an irreducible polynomial $p(x)\in K[x]$ and a closed path $\pi \in E$ based on the vertex $v$,   Anh and Nam    in  \cite{AN20}   constructed   a further simple module, named Rangaswamy module: it is the   quotient of  the projective left ideal of $L_K(E)$ generated by the vertex $v$ by the submodule generated by the polynomial $p(x)$  evaluated in $\pi$.

\smallskip

In \cite{AMT15}, the authors of this note, in collaboration with Gene Abrams, provided explicit projective resolutions for the Chen simple modules \R{and studied their extension vector spaces. It turns out that the dimension of the $\Ext$-vector spaces between two Chen simple modules can be easily handled in terms of features of the graph $E$. These  results have been useful to deeply study the module category  of $L_K(E)$, as for instance in describing Pr\"ufer modules \cite{AMT19} or the injective envelopes of simple modules when $E$ is the Toeplitz graph (i.e $L_K(E)$ is isomorphic to  the Jacobson algebra) \cite{AMT21}, or when $E$ is a finite graph with disjoint cycles \cite{AMT24}}. 

\smallskip

In this paper  we aim to construct the projective resolutions \R{of a family of simple modules $V^p_{[e^\infty]}$, constructed as in \cite{AR14} but allowing $e$ to be any cycle in any graph $E$, and to study their Ext-groups.  The established constructions and results  will be used in a forthcoming paper dedicated to studying the injective envelopes of simple modules over Leavitt path algebras of arbitrary graphs \cite{AMT26}.  }

\section{A technical lemma}\label{sec:Extending scalars}
Let $K$ be a field and $q(x)=q_nx^n+q_{n-1}x^{n-1}+\cdots+q_1x+q_0$ an irreducible polynomial. Consider the field extension $K'=K[x]/\langle q(x)\rangle$ and denote by $\x$ the coset $\x+\langle q(x)\rangle\in K'$. Starting from any $K$-algebra $A$ (not necessarily with identity), one can construct the $K'$-algebra $A':=A\otimes_K K'$. Clearly $A\cong A\otimes_K K$ can be identified with a $K$-subalgebra of $A'$.

Every element $\mathfrak b\in A'$ can be written in a unique way in the form 
\[\mathfrak b=b_0\otimes 1+b_1\otimes \x+\cdots+b_{n-1}\otimes \x^{n-1}\]
for suitable $b_i\in A$, $i=0,..., n-1$. Such an element $\mathfrak b$ belongs to $A$ if and only if $b_1=b_2=\cdots=b_{n-1}=0$.

Let us fix two elements $e,v$ in $A$ satisfying
\[v^2=v,\quad ev=e=ve.\]
For each polynomial $f(x)=k'_\ell x^\ell+\cdots+k'_1x+k'_0\in K'[x]$ we define the evaluation of $f$ in $e$ setting
\[f(e)=k'_\ell e^\ell+\cdots+k'_1e+k'_0v\in A'.\]
The map sending any polynomial in $K'[x]$ to its valuation in $e$ is an homomorphism of $K'$-algebras, 
which restricts to an homomorphism of $K$-algebras $K[x]\to A$.

In this preliminary section we present a general result we will use in Section 4 on the relationship between the left $A$-ideal $Aq(e)$ and the contraction $A'(e-\x v)\cap A$ of the left $A'$-ideal $A'(e-\x v)$. The proof is due to Daniela La Mattina of the University of Palermo.

%
%If $q(x)=q_mx^m+\cdots+q_1x+q_0$ is any polynomial in $K[x]$ we define the evaluation $q(e)$ of $q(x)$ in $e$ setting
%\[q(e):= q_me^m+\cdots+q_1e+q_0v\in A.\]
%The map $\nu_K(e):K[x]\to A$, $\nu_K(e)(q(x))=q(e)$, is an homomorphism of $K$-algebras.

\begin{lemma}\label{thm:main}
It is $A'(e-\x v)\cap A=\R{Aq(e)}$.
\end{lemma}
\begin{proof}
Since $q(\x)=0$ and $x-\x$ is the minimal polynomial of $\x$ in $K'[x]$ we have $q(x)=r(x)(x-\x)$ for a suitable $r(x)\in K'[x]$. Therefore in $A'$ we have
$q(e)=r(e)(e-\x v)$ and hence $Aq(e)\subseteq A'(e-\x v)\cap A$.\\
Consider an element $\mathfrak a$ of the intersection $A'(e-\x v)\cap A$; then we have
\[\mathfrak a\otimes 1=\mathfrak a=\left(\sum_{i=0}^{n-1}b_i\otimes \x^i\right)(e-\x v)\]
with $b_0,...,b_{n-1}\in A$. Since $0=q(\x)=\x^n+q_{n-1}\x+\cdots+q_1\x+q_0$, we have 
\[-\x^n=q_{n-1}\x+\cdots+q_1\x+q_0.\]
Then
\begin{align*}
\mathfrak a=&\left(\sum_{i=0}^{n-1}b_i\otimes \x^i\right)(e\otimes 1- v\otimes\x)\\
=&\sum_{i=0}^{n-1}b_ie\otimes \x^i-\sum_{i=0}^{n-1}b_iv\otimes \x^{i+1}\\
=&\sum_{i=0}^{n-1}b_ie\otimes \x^i-\sum_{i=0}^{n-2}b_iv\otimes \x^{i+1}+b_{n-1}v\otimes (q_{n-1}\x^{n-1}+\cdots+q_1\x+q_0)\\
=&(b_0e+b_{n-1}vq_0)\otimes 1+(b_1e-b_0v+b_{n-1}vq_1)\otimes\x+\cdots+\\
&+(b_{n-1}e-b_{n-2}v+b_{n-1}vq_{n-1})\otimes\x^{n-1}.
\end{align*}
Since $\mathfrak a\in A$, we have
\begin{align*} 
\mathfrak a&=b_0e+b_{n-1}vq_0\\
0&=b_1e-b_0v+b_{n-1}vq_1\\
&\cdots\\
0&=b_{n-1}e-b_{n-2}v+b_{n-1}vq_{n-1}.
\end{align*}
Since $b_ie=(b_iv)e=(b_{i+1}e+b_{n-1}vq_{i+1})e$  for $0\leq i<n-1$, we get
\begin{align*} 
\mathfrak a&=b_0e+b_{n-1}vq_0\\
&=(b_1e+b_{n-1}vq_1)e+b_{n-1}vq_0=b_1e^2+b_{n-1}(q_1e+q_0v)\\
&\cdots\\
&=b_{n-1}e^{n}+b_{n-1}(q_{n-1}e^{n-1}+\cdots+q_1e+q_0v)\in Aq(e). \qedhere
\end{align*}
\end{proof}

\section{Leavitt path algebras}

A (directed) graph $E = (E_0, E_1, s, r)$ consists of a vertex set $E_0$, an edge set
$E_1$, and source and range functions $s, r : E_1\to E_0$. For $v \in E_0$, the set of edges
$\{e \in E_1 \mid s(e) = v\}$ is denoted $s^{-1}(v)$. 
%The graph $E$ is called finite in case both $E0 and E1
%are finite sets. 
A path $\alpha$ in $E$ is a sequence $e_1e_2\cdots e_n$ of edges in $E$ for which
$r(e_i) = s(e_{i+1})$ for all $1\leq  i\leq  n-1$. We say that such $\alpha$ has length $n$, and we write
$s(\alpha) = s(e_1)$ and $r(\alpha) = r(e_n)$. We view each vertex $v \in E_0$ as a path of length 0,
and set $v = s(v) = r(v)$. We denote the set of paths in $E$ by Path$(E)$. 
\R{A path $\lambda$ in Path$(E)$ is left (right) divisible by a path $\mu$ if $\lambda=\mu \lambda' $ ($\lambda=\lambda' \mu$) for a suitable $\lambda'\in \text{Path}(E)$.  We write $\mu | \lambda$ if $\mu$ divide on the left $\lambda$, and we do away with a specific notation for right divisibility.}

A path $\gamma= e_1e_2\cdots e_n$,  ($n \geq 1$) in $E$ is 
\begin{itemize}
\item \emph{closed} if
$r(e_n) = s(e_1)$.
%\item \emph{simple closed} if it is closed and $s(e_i)\not=s(e_1)$ for all $1<i\leq n$.
\item \emph{a cycle} if it is closed and $s(e_i)\not=s(e_j)$ for all $i\not=j$.
\item \emph{an exclusive cycle} if it is a cycle and it is disjoint with every other cycle; equivalently,
no vertex $u$ on $\gamma$ is the base of a different cycle other than the cyclic permutation of $\gamma$ based
at $u$. 
\end{itemize}

The path algebra $KE$ is the $K$-vector space with basis Path$(E)$ with the multiplication induced by the concatenation of paths.

We denote by $\widehat E$ the \emph{double graph} of $E$, gotten by adding to $E$ an edge $e^*$ in the reversed direction for each edge $e\in E_1$: $s(e^*)=r(e)$, and $r(e^*)=s(e)$.

The \emph{Leavitt path algebra} $L_K(E)$ is the path algebra $K\widehat E$ modulo the Cuntz-Krieger relations
\begin{enumerate}
\item[CK1] $e^*e'=\delta_{e,e'}r(e)$ for all $e,e'\in E^1$,
\item[CK2] $v=\displaystyle\sum_{e\in s^{-1}v}ee^*$ for every $v\in E^0$ for which $0<|s^{-1}v|<\infty$.
\end{enumerate}
 Every element of $L_K(E)$ may be written as $\sum_{i=1}^n k_i\alpha_i\beta_i^*$ where $k_i$ is a non-zero element of $K$, and $\alpha_i$, $\beta_i$ are paths in $E$.
 
 The set of all finite sums of vertices is a set of  \emph{local units} for $L_K(E)$. 
 An unital left $L_K(E)$-module is an abelian group $M$ with a (standard) module action of $L_K(E)$ on $M$, but with the added proviso that $L_K(E)M=M$. 
If $M$ is a unital left $L_K(E)$-module, then for each $m\in M$ there exists $\ell_m\in L_K(E)$ and $m'\in M$ such that $m=\ell_m m'$. It is easy to check that it is possible to assume $m=m'$ and $\ell_m$ equal to a finite sum of vertices. Indeed there exists a finite set $F_m$ of vertices such that $\sum_{v\in F_m}v\ell_m=\ell_m$ and hence
\[\left(\sum_{v\in F_m}v \right)m=\left(\sum_{v\in F_m}v \right)\ell_m m'=\left(\sum_{v\in F_m}v\ell_m\right) m'=\ell_m m'=m.\]
 
In the sequel with the term ``module'' we will always mean a unital module. 

\R{We recommend that readers wishing to deepen their understanding of
Leavitt algebras consult the seminal monograph \cite{AAM}.}

\section{A projective resolution of $V^{\sigma_{a,e}}_{[e^\infty]}$}

For each cycle $e$ in $E$ we can consider the infinite path $e^\infty$ obtained circling around $e$ infinitely many times. An infinite path $\mathfrak p$ in $E$ is tail equivalent to $e^\infty$ if $\mathfrak p=\lambda e^\infty$ for a suitable path $\lambda\in\Path(E)$. We denote by $[e^\infty]$ the set of all infinite paths tail equivalent to $e^\infty$. 
Consider the $K$-vector space $V_{[e^\infty]}$ having as $K$-basis the set $[e^\infty]$. For any $v\in E^0$, $f\in E^1$ and each $\mathfrak q=f_1\cdots f_n e^\infty$, $n\geq 2$, in $[e^\infty]$, define
\[v \mathfrak q=\begin{cases}
    \mathfrak q =f_1\cdots f_n e^\infty& \text{if } v=s(f_1), \\
     0 & \text{otherwise},
\end{cases}\]
\[f \mathfrak q=\begin{cases}
    f\mathfrak q=ff_1\cdots f_n e^\infty  & \text{if } r(f)=s(f_1), \\
     0 & \text{otherwise},
\end{cases}\quad\text{and}\]
\[f^* \mathfrak q=\begin{cases}
    f_2\cdots f_n e^\infty  & \text{if }f=f_1, \\
     0 & \text{otherwise}.
\end{cases}\]
The $K$-linear extension to all of $V_{[e^\infty]}$ of this action endow $V_{[e^\infty]}$ with a structure of a left $L_K(E)$-module. 

 In \cite[Theorem 3.3.]{Ch15}, Chen proved that $V_{[e^\infty]}$ is a simple left $L_K(E)$-module. He also introduced the twisted version of this type of modules, obtaining a new family of simples. 
 Let $a \in K$ with $a \neq 0$. We denote by $\sigma_{e,a}$ the gauge automorphism of $L_K(E)$ associated with $e$ and $a$. Then $\sigma_{e,a}$ maps each vertex $u \in E^0$, each edge $f \in E^1 \setminus {e_1}$, and, similarly, each $f^* \in (E^1)^* \setminus {e_1^*}$ to itself. However, it sends $e_1$ to $a e_1$ and $e_1^*$ to $a^{-1} e_1^*$. The automorphism depends only on the first edge $e_1$ of the cycle $e$. 
In the special case where $a = 1$, the automorphism $\sigma_{e,1}$ clearly reduces to the identity map on $L_K(E)$.

For $M \in L_K(E)\text{-Mod}$, we define the twisted left $L_K(E)$-module $M^{\sigma_{e,a}}$ by taking $M^{\sigma_{e,a}} := M$ as an abelian group, but modifying the left $L_K(E)$-action via
$$\ell\star m:=\sigma_{e,a}(\ell)m$$
for all $\ell \in L_K(E)$ and $m \in M$. Since $M$ is a unital left $L_K(E)$-module, we have $L_K(E)M=M$; therefore also $L_K(E)M^{\sigma_{e,a}}=M^{\sigma_{e,a}}$ is a unital module. The twisted left $L_K(E)$-module $\left({}_{L_K(E)}L_K(E)\right)^{\sigma_{e,a}}$
of  the left regular $L_K(E)$-module has also a structure of right $L_K(E)$-module, with the usual multiplication in $L_K(E)$. The two structures are compatible and therefore $\left(L_K(E)\right)^{\sigma_{e,a}}$ is a $L_K(E)$-bimodule. Indeed we have for $\lambda_1$,  $\lambda_2\in L_K(E)$, and $\ell\in L_K(E)^{\sigma_{e,a}}$
\[\lambda_1\star(\ell\lambda_2)=\sigma_{e,a}(\lambda_1)(\ell\lambda_2)=(\sigma_{e,a}(\lambda_1)\ell)\lambda_2=(\lambda_1\star\ell)\lambda_2.\]

Consider the family $\mathcal F$ of the finite sums of vertices. We can define a partial order $\leq$  on $\mathcal F$ setting $e\leq f$ if $ef=fe=e$ for each $e,f\in \mathcal F$. For each $e\in \mathcal F$, the corner ring $eL_K(E)e$ is a ring with unit $e$. For each left $L_K(E)$-module $M$ and any $e\in\mathcal F$ the abelian subgroup $eM$ of $M$ is a left $eL_K(E)e$-module. 
\begin{proposition}\cite[Lemma 1.5]{Ab83}\label{prop:tensor}
For each left $L_K(E)$-module $M$ we have the following isomorphism of abelian groups and left $L_K(E)$-modules
\[\varinjlim_{e\in\mathcal F}(L_K(E)e\otimes_{eL_K(E)e} eM)\cong M.\]
\end{proposition}
In the sequel we will set 
\[L_K(E)\ot_{L_K(E)} M:=\varinjlim_{e\in\mathcal F}(L_K(E)e\otimes_{eL_K(E)e} eM).\]

\begin{lemma}\label{lemma:sigmafunctor}
If $M$ is a left $L_K(E)$-module, then the twisted left $L_K(E)$-module $M^{\sigma_{e,a}}$ is isomorphic to the tensor product  $\left(L_K(E)\right)^{\sigma_{e,a}}\ot_{L_K(E)}M$ of the $L_K(E)$-bimodule $\left(L_K(E)\right)^{\sigma_{e,a}}$ and the left $L_K(E)$-module $M$.
\end{lemma}
\begin{proof}
By Proposition~\ref{prop:tensor} we have the following isomorphism of abelian groups:
\[L_K(E)\ot_{L_K(E)} M\cong M.\]
Then we get easily the following isomorphisms of left $L_K(E)$-modules
\[M^{\sigma_{e,a}}\cong \left(L_K(E)\ot_{L_K(E)}M\right)^{\sigma_{e,a}}\cong L_K(E)^{\sigma_{e,a}}\ot_{L_K(E)}M.\qedhere\]
\end{proof}
As consequence of Lemma~\ref{lemma:sigmafunctor}, we get that the automorphism $\sigma_{e,a}$ determines an auto-equivalence of the category $L_K(E)\text{-Mod}$, given by the functor
 \[L_K(E)^{\sigma_{e,a}}\ot_{L_K(E)}-: \ \ L_K(E)\text{-Mod}\to L_K(E)\text{-Mod}\]
 which assigns to each $L_K(E)$-module $M$ the  module $ L_K(E)^{\sigma_{e,a}} \ot M\cong M^{\sigma_{e,a}}$. 
 In particular for every $a \in K$ with $a \neq 0$, the twisted module $V^{\sigma_{e,a}}_{[e^\infty]}$ is again a simple left $L_K(E)$-module.

%\section{A projective resolution of $V_{[e^\infty]}^{\sigma_{e,a}}$}
%
%In a ring with local units $R$, for each idempotent $u\in R$, the left $R$-module $Ru$ is a finitely generated unitary projective module.

In \cite[Theorem 2.8]{AMT15} we proved, under the hypothesis that $E$ is a finite graph, that $V_{[e^\infty]}$ is finitely presented and that, denoted by $\rho_{e-s(e)}$ and $\rho_{e^\infty}$ the right multiplication by $e-s(e)$ and $e^\infty$, respectively, a projective resolution of $V_{[e^\infty]}$ is given by
\[\xymatrix{0\ar[r]&L_K(E)s(e)\ar[r]^{\rho_{e-s(e)}}&L_K(E)s(e)\ar[r]^-{\rho_{e^\infty}}&V_{[e^\infty]}\ar[r]&0}.\]
The same proof works for an arbitrary graph $E$: observe that since $s(e)$ is an idempotent element in $L_K(E)$, the left $L_K(E)$-module $L_K(E)s(e)$ is a unital projective module.
The previous exact sequence is equivalent to 
\[\xymatrix{0\ar[r]&L_K(E)(e-s(e))\ar@{^(->}[r]^-{\iota}&L_K(E)s(e)\ar[r]^-{\rho_{e^\infty}}&V_{[e^\infty]}\ar[r]&0}.\]

Applying the auto-equivalence $L_K(E)^{\sigma_{e,a}}\ot_{L_K(E)}-$ we get the short exact sequence of left $L_K(E)$-modules
\[\xymatrix{0\ar[r]&L_K(E)(e-s(e))^{\sigma_{e,a}}\ar@{^(->}[r]^-{\iota}&(L_K(E)s(e))^{\sigma_{e,a}}\ar[r]^-{\rho_{e^\infty}}&V_{[e^\infty]}^{\sigma_{e,a}}\ar[r]&0}.\]

The automorphism $\sigma_{e,a}$ of the algebra $L_K(E)$ induces the isomorphism of left $L_K(E)$-modules
\[L_K(E)s(e)\to \left(L_K(E)s(e)\right)^{\sigma_{e,a}},\quad \ell s(e)\mapsto \sigma_{e,a}(\ell s(e))=\sigma_{e,a}(\ell)s(e)=\ell\star s(e).\]
We continue to denote it by $\sigma_{e,a}$. It induces the isomorphism
 \begin{align*}\psi:L_K(E)(a^{-1}&e-s(e))\to L_K(E)(e-s(e))^{\sigma_{e,a}},
 \\
 &\ell (a^{-1}e-s(e))\mapsto \sigma_{e,a}(\ell (a^{-1}e-s(e)))=\sigma_{e,a}(\ell)(e-s(e)).
% 
% \ell(a^{-1}e-s(e))\star s(e)=\ell\star (e-s(e))=(\ell(e-s(e)))^{\sigma_{e,a}}.
 \end{align*}
 Since
 \[\iota\circ\psi(\ell(a^{-1}e-s(e)))=\sigma_{e,a}\circ\iota(\ell(a^{-1}e-s(e)))\]
 we get the following commutative diagram:
 \[
 \xymatrix{0\ar[r]&L_K(E)(e-s(e))^{\sigma_{e,a}}\ar@{^(->}[r]^-{\iota}&(L_K(E)s(e))^{\sigma_{e,a}}\ar[r]^-{\rho_{e^\infty}}&V_{[e^\infty]}^{\sigma_{e,a}}\ar[r]&0\\
 0\ar[r]&L_K(E)(a^{-1}e-s(e))\ar@{^(->}[r]^-{\iota}\ar[u]^\psi&L_K(E)s(e)\ar[r]^-{\rho_{e^\infty}{ \circ}\sigma_{e,a}}\ar[u]^{\sigma_{e,a} }&V_{[e^\infty]}^{\sigma_{e,a}}\ar[r]\ar@{=}[u]&0
 } \]
We have proved
 \begin{theorem}\label{thm:parte1}
 Let $E$ be any graph. Let $e$ be a simple closed path in $E$, and $0\not=a\in K$. Then the Chen simple module $V_{[e^\infty]}^{\sigma_{e,a}}$ is finitely presented and a projective
resolution of $V_{[e^\infty]}^{\sigma_{e,a}}$ is given by
\[\xymatrix{0\ar[r]& L_K(E)(a^{-1}e-s(e))
 \ar[r]^-{ \iota}&L_K(E)s(e)\ar[rr]^-{\rho_{e^\infty}\circ\sigma_{e,a}}&&V_{[e^\infty]}^{\sigma_{e,a}}\ar[r]&0.}
 \]
 %where $\rho^\sigma_{e^\infty}(\ell s(e))=\sigma_{e,a}(\ell)e^\infty$.
 \end{theorem}

\section{A projective resolution of  $V^p_{[e^\infty]}$}

In \cite{AR14} Ara and Rangaswamy generalised further the construction of simple modules through the twisting associated to a cycle and a non-zero element of an algebraic simple extension of $K$. Let $p(x) \in K[x]$ be a \emph{basic irreducible} polynomial of degree $\geq 2$—meaning that $p(x)$ is irreducible in $K[x]$, $\deg p(x)\geq 2$,  and it satisfies $p(0) = -1$. 
Assume $p(x)=p_nx^n+\cdots+p_1x-1$. Define $K' := K[x]/\langle p(x) \rangle$ \R{and let $\overline{x}_p$, or simply $\x$ when the polynomial $p(x)$ is clear from the context,} denote the coset $x + \langle p(x) \rangle \in K'\setminus K$. Clearly $\overline{x}\not=0$ and hence is invertible in $K'$: 
\[\overline{x}^{-1}=p_n\overline x^{n-1}+\cdots+p_2\overline x+p_1.\]
Using the gauge automorphism $\sigma_{e,\overline x}$ of $L_{K'}(E)$, the Leavitt path algebra associated to the graph $E$ and the extension $K'$ of $K$, we can consider the simple left $L_{K'}(E)$-module $V^{\sigma_{e,\overline x}}_{[e^\infty]}$.
\begin{definition}
We define $V^p_{[e^\infty]}$ to be the left $L_K(E)$-module obtained by restricting scalars from $K'$ to $K$ on the twisted left $L_{K'}(E)$-module $V^{\sigma_{e,\overline{x}}}_{[e^\infty]}$.
\end{definition}
The following result   slightly  generalises \cite[Lemma~3.3]{AR14}, following essentially the same proof, to any cycle in any graph $E$:
\begin{proposition}\label{prop:Vpsimple}
The left $L_K(E)$-module $V^p_{[e^\infty]}$ is simple.
\end{proposition}
\begin{proof}
Let $U$ be a non zero left $L_K(E)$-submodule of $V^p_{[e^\infty]}$. First, let us prove that if $\lambda e^\infty\in U$ with $0\not=\lambda\in K'$, then also $\lambda' e^\infty\in U$ for each $\lambda'\in K'$. Since $U$ is a $K$-vector space, it is enough to prove that $\x \lambda e^\infty\in U$. This follows by
\[U\ni e\star \lambda e^\infty=\x\lambda e^\infty.\]
Let $\sum_{i=1}^\ell \lambda_i\mu_ie^\infty$ be a non zero element in $U$ with $0\not=\lambda_i\in K'$ and the $\mu_i$ distinct paths not right divisible by $e$. Without loss of generality, we can assume that the length of $\mu_1$ is larger than or equal to the length of all the other paths $\mu_2$, ..., $\mu_\ell$. Denote by $\#_e(\mu_1)$ the number of occurrences of $e_1$ in $\mu_1$. We have
\[U\ni \mu_1^*\star \left(\sum_{i=1}^\ell \lambda_i\mu_ie^\infty\right)=\x^{-\#_e(\mu_1)}\lambda_1e^\infty\]
with $0\not= \x^{-\#_e(\mu_1)}\lambda_1\in K'$.
Then for what we proved at the beginning of this proof we have that $\lambda'e^\infty\in U$ for each $\lambda'\in K'$.
Consider now any infinite path $\nu e^\infty$ with $\nu$ finite path not right divisible by $e$, and any $\lambda'\in K'$. Then 
\[\lambda'\nu e^\infty=\nu \lambda' e^\infty=\nu (e^*)^{\#_e(\nu)}\star \lambda' e^\infty\in U;\]
hence $U$ contains all the $K'$-multiples of infinite paths tail equivalent to $e^\infty$ and therefore $U=V^p_{[e^\infty]}$.
%Since $U$ contains
%\[(\nu (e^*)^{\#_e(\mu_1)})\star \lambda_1e^\infty=\lambda_1 \nu e^\infty\in L.\]
%Finally, to prove that $k' \nu e^\infty\in L$ for each $k'\in K'$, being $L$ a $K$-vector space it is enough to prove that $\x(\lambda_1\nu e^\infty)$ belongs to $L$. This follows by
%\[(\nu e^{\#_e(\nu)+1} \nu^*)\star \lambda_1 \nu e^\infty\]
%
%$\mathfrak p=f_1\cdots f_\ell e^\infty$ be any infinite path in $[e^\infty]$. It is not restrictive to assume $\ell\geq N$ and $f_{\ell-N+1}=e_1$, ..., $f_\ell=e_N$. Denote by $m:=\min\{i\mid f_i=e_1, 1\leq i\leq \ell\}$. Then, in the left $L_{K'}(E)$-module $V^{\sigma_{e,\overline x}}_{[e^\infty]}$, we have
%\[\overline x \mathfrak p =\overline x f_1\cdots f_{m-1}e_1 f_{m+1}\cdots f_\ell e^\infty=
%f_1\cdots f_{m-1}e_1\star  f_{m+1}\cdots f_\ell e^\infty.\]
%Therefore any $L_K(E)$-submodule of $V^p_{[e^\infty]}$ is also an $L_{K'}(E)$-submodule of $V^{\sigma_{e,\overline x}}_{[e^\infty]}$. Since the latter is simple, also $V^p_{[e^\infty]}$ is simple.
\end{proof}

%The set of all regular monomial, i.e., the set $\mathcal B$ of all paths $\lambda \mu^*$ such that $\lambda,\mu\in \text{Path}(E)$ and $r(\lambda)=r(\mu)$ is a $K$-basis for the Cohn algebra $C_K(E)$ \cite[Proposition 1.5.6]{AAM}. For each regular vertex $v\in \text{Reg}(E)\subseteq E^0$, let $\{\varepsilon^v_1, \dots,\varepsilon^v_{n_v}\}$ be an enumeration of the elements of $s^{-1}(v)$. Then 
%\begin{lemma} \cite[Corollary 1.5.12]{AAM}\label{lemma:defbase}
%The set
%\[\mathcal B'':=\mathcal B\setminus\{\lambda\varepsilon^v_{n_v}(\varepsilon^v_{n_v})^*\mu^*\mid r(\lambda)=r(\mu)=v\in \text{Reg}(E)\}\]
%is a $K$-basis for $L_K(E)$.
%\end{lemma}
%Moreover one has:
%\begin{lemma}\label{lemma:base}
%If $e=e_1\cdots e_n$ is a real cycle in $E$ and $\lambda \mu^*$ is an element of the basis $\mathcal B''$  of $L_K(E)$, then 
%\[\lambda \mu^*e\not=0\Rightarrow \lambda \mu^*e=\lambda'{\mu'}^*\in\mathcal B''.
%\]
%\end{lemma}
%\begin{proof}
%The path $\lambda \mu^*e$ is not 0 if an only if either $\mu$ is an initial part of $e$ or $e$ is an initial part of $\mu$:
%\[\mu=e_1\cdots e_\ell,\ 1\leq\ell\leq n,\quad\text{or}\quad \mu=e_1\cdots e_n\mu'.\]
%In the first case $\lambda \mu^*e=\lambda e_{\ell+1}\cdots e_n\in\mathcal B''$; in the second case
%$\lambda \mu^*e=\lambda{\mu'}^*\in\mathcal B''$.
%\end{proof}

\begin{theorem}\label{thm:Nam}
Let $e=e_1\cdots e_N$ be a cycle in $E$, and $p(x) \in K[x]$ a basic irreducible polynomial of degree $n\geq 2$. Then the Chen simple module $V_{[e^\infty]}^{p}$ is finitely presented and a projective
resolution of $V_{[e^\infty]}^{p}$ is given by
\[\xymatrix{0\ar[r]& L_K(E)p(e)
 \ar[r]^-{ \iota}&L_K(E)s(e)\ar[rr]^-{\rho_{e^\infty}\circ\sigma_{e,\x}\circ i}&&V_{[e^\infty]}^{p}\ar[r]&0}
 \]
 where $i:L_K(E)s(e)\to L_{K'}(E)s(e)$ is the natural embedding of left $L_K(E)$-modules.
\end{theorem}
\begin{proof}
By Theorem~\ref{thm:parte1} we have the following short exact sequence of $L_{K'}(E)$-modules
\[\xymatrix{0\ar[r]& L_{K'}(E)(\overline x^{-1}e-s(e))
 \ar[r]^-{ \iota}&L_{K'}(E)s(e)\ar[rr]^-{\rho_{e^\infty}\circ\sigma_{e,\x}}&&V_{[e^\infty]}^{\sigma_{e,\overline x}}\ar[r]&0.}
 \]
 Observed that $L_{K'}(E)(\overline x^{-1}e-s(e))=L_{K'}(E)(e-\overline xs(e))$, restricting the scalars from $K'$ to $K$, we have the following commutative diagram of left $L_{K}(E)$-modules
 \[\xymatrix{0\ar[r]& L_{K'}(E)(e-\overline xs(e))
 \ar[r]^-{ \iota}&L_{K'}(E)s(e)\ar[rr]^-{\rho_{e^\infty}\circ\sigma_{e,\x}}&&V_{[e^\infty]}^{p(x)}
 \ar[r]&0\\
 0\ar[r]& L_{K'}(E)(e-\overline xs(e))\cap L_K(E)s(e)\ar@{^(->}[u]
 \ar[r]^-{\iota}&L_K(E)s(e)\ar@{^(->}[u]^i
 %\ar[rr]^-{\rho_{e^\infty}}&&V_{[e^\infty]}^{p(x)}\ar@{=}[u]\ar[r]&0
 }
 \]
 Since by Proposition~\ref{prop:Vpsimple} $V_{[e^\infty]}^{p}$ is a simple $L_K(E)$-module and 
 $(\rho_{e^\infty}\circ\sigma_{e,\x}\circ i)(s(e))=e^\infty\not=0$, the homomorphism $\rho_{e^\infty}\circ\sigma_{e,\x}\circ i$ is surjective and we have the short exact sequence of left $L_K(E)$-modules
 \[\xymatrix{0\ar[r]& L_{K'}(E)(e-\overline xs(e))\cap L_K(E)s(e)\
 \ar[r]^-{\iota}&L_K(E)s(e)\ar[rr]^-{\rho_{e^\infty}\circ\sigma_{e,\x}\circ i}&&V_{[e^\infty]}^{p(x)}
 \ar[r]&0}\]
By \Cref{thm:main} $L_K(E)s(e)\cap L_{K'}(E)(e-\x s(e))=L_K(E)p_n^{-1}p(e)=L_K(E)p(e)$.
\end{proof}

\begin{remark}
\Cref{thm:Nam} proves $V_{[e^\infty]}^{p}\cong L_K(E)s(e)/L_K(E)p(e)$. In case $e$ is an exclusive cycle, this show explicitly that the simple modules introduced by Ara and Rangaswamy in \cite{AR14} are isomorphic to those studied by Anh and Nam in \cite{AN20}.
\end{remark}
\section{Extensions}

In this section we apply the previous results to compute the $K$-dimension of the vector space $\Ext(V_{[e^\infty]}^{p}, V_{[c^\infty]}^{q})$ of the extensions of two simple modules associated to cycles $e$ and $c$, and basic irreducible polynomials $p(x)$ and $q(x)$, respectively.

Let us consider the short exact sequence
\[\xymatrix{0\ar[r]&L_K(E)s(e)\ar[rr]^{\pi_{p(e)}}&&L_K(E)s(e)\ar[r]&V^{p}_{[e^\infty]}\ar[r]&0}\]
where $\pi_{p(e)}$ is the right product by $p(e)$.
Since $L_K(E)s(e)$ is a projective left $L_K(E)$-module,  applying the contravariant functor $\Hom_{L_K(E)}(-,V_{[c^\infty]}^{q})$ we get the following exact sequence of $K$-vector spaces
\begin{align*}
\xymatrix{\Hom_{L_K(E)}(L_K(E)s(e),V_{[c^\infty]}^{q})\ar[r]^{\pi_{p(e)}^*}&\Hom_{L_K(E)}(L_K(E)s(e),V_{[c^\infty]}^{q})\ar@{-}[r]^-\psi&{}}\\
\xymatrix{{}\ar[r]&\Ext^1(V^{p}_{[e^\infty]},V_{[c^\infty]}^{q})\ar[r]&0,}
\end{align*}
where $(\pi_{p(e)}^*(\varphi))(s(e))=(\varphi\circ \pi_{p(e)})(s(e))=\varphi(p(e))=p(e)\star \varphi(s(e))$. If $(e,p(x))\not=(c,q(x))$ then $V^{p}_{[e^\infty]}$ and $V_{[c^\infty]}^{q}$ are non isomorphic simple modules and therefore $\pi^*_{p(e)}$ is a monomorphism.

For any left $L_K(E)$-module $M$, the abelian group $\Hom(L_K(E)s(e),M)$ is isomorphic to
$s(e) M:=\{m\in M:s(e)m=m\}$: indeed, $s(e)\mapsto s(e)m$ defines an element of $\Hom(L_K(E)s(e),M)$ for any $ m\in s(e)M$. Therefore we get the exact sequence of $K$-vector spaces
\[\xymatrix{s(e)V_{[c^\infty]}^{q}\ar[rr]^{p(e)\star-}&& s(e)V_{[c^\infty]}^{q}\ar[r]&\Ext^1(V^{p}_{[e^\infty]},V_{[c^\infty]}^{q})\ar[r]&0}\]
where the morphism of $K$-vector spaces  $p(e)\star-$ is the left multiplication by $p(e)$ in $s(e)V_{[c^\infty]}^{q}$:
\[p(e)\star s(e)\lambda=(p(e))^{\sigma_{c,\x_q}}.\]
Clearly, if $s(e)V_{[c^\infty]}^{q}=0$, i.e., there are no real paths starting in $s(e)$ and ending in $s(c)$,  we get immediately $\Ext^1(V^{p}_{[e^\infty]},V_{[c^\infty]}^{q})=0$ for all polynomials $p(x)$ and $q(x)$. Therefore  in the sequel we can assume that $s(e)\Path(E)s(c)\not=\emptyset$. Let us consider the non empty set
\[
L_{(e,c)}:=\{\lambda\in E^\infty\mid \lambda\sim c^\infty, e\not|\,\lambda, s(\lambda)=s(e)\}.\]

If the cycles $c$ and $e$ are exclusive, we can compute the $K$-dimension of $\Ext^1(V^{p}_{[e^\infty]},V_{[c^\infty]}^{q})$  in terms of the cardinality of  $L_{(e, c)}$ and of the degrees of the polynomials $p(x)$ and $q(x)$.

\begin{theorem}\label{thm:ext}
Let $c$ and $e$ be two exclusive cycles in the graph $E$, and $p(x)$ and $q(x)$ two basic irreducible polynomials in $K[x]$. Then
\[
\dim_K \Ext^1(V^p_{[c^\infty]}, V^q_{[e^\infty]})=\begin{cases}
 |L_{e,c}|\times\deg p(x)\times\deg q(x)     & \text{if }c\not=e, \\
 0    & \text{if }c=e,\text{ and }p(x)\not=q(x) \\
 \deg p(x)
      & \text{if }c=e,\text{ and }p(x)=q(x).
\end{cases}
\]
\end{theorem}
\begin{proof}
\begin{description}
\item[$c\not=e$] A $K$-basis for $s(c)V_{[e^\infty]}^{q}$ is given by the infinite paths
\[\x_q^j c^i\lambda e^\infty \quad 0\leq j<\deg q(x), i\geq 0, \lambda\in L_{c,e}.\]
Considering the quotient with respect to $p(c)\star V_{[e^\infty]}^{q}=p(c)V_{[e^\infty]}^{q}$ we get that 
\[\{\x_q^jc^i\lambda e^\infty+p(c)V_{[e^\infty]}^{q}\mid 0\leq j<\deg q(x), 0\leq i<\deg p(x), \lambda\in L_{c,e}\}\]
is a $K$-basis for $s(e)V_{[c^\infty]}^{q}/p(c)\star V_{[e^\infty]}^{q}\cong \Ext(V^p_{[c^\infty]}, V^q_{[e^\infty]})$. 
\item[$c=e$ and $p(x)\not=q(x)$] Since both $p(x)$ and $q(x)$ are irreducible, they are relatively prime in $K[x]$ and hence there exist $\alpha(x),\beta(x)\in K[x]$ such that $p(x)\alpha(x)+\beta(x)q(x)=1$. A $K$-basis for $s(c)V_{[c^\infty]}^{q}$ is given by the infinite paths
\[\x_q^j c^\infty \quad 0\leq j<\deg q(x).\]
Since
\[\x_q^j c^\infty=\x_q^j s(c) c^\infty=\x_q^j (p(c)\alpha(c)+\beta(c)q(c))c^\infty=p(c)\x_q^j c^\infty\in p(c)\star V_{[c^\infty]}^{q}\]
we get that $\Ext(V^p_{[c^\infty]}, V^q_{[c^\infty]})=0$.
\item[$c=e\text{ and }p(x)=q(x)$] A $K$-basis for $s(c)V_{[c^\infty]}^{p}$ is given by the infinite paths
\[\x_p^j c^\infty \quad 0\leq j<\deg p(x).\]
Since $p(c)\star V_{[c^\infty]}^{p}=0$ we conclude that the $K$-dimension of the space $\Ext(V^p_{[c^\infty]}, V^p_{[c^\infty]})$ is $\deg p$.
\end{description}
\end{proof}

If at least one of  the cycles $c$ and $e$ is  not exclusive,  then $L_{(e, c)}$ is an infinite set. In this case we show that $\dim_K \Ext^1(V^p_{[c^\infty]}, V^q_{[e^\infty]})=\infty$. The crucial point is the following

\begin{lemma}\label{lemma:non excl}
In the quotient $s(e)V_{[c^\infty]}^{q}/p(e)\star V_{[c^\infty]}^{q}$, the cosets $\lambda+p(e)\star V_{[c^\infty]}^{q}$, $\lambda\in L_{(e,c)}$, are $K$-linearly independent.
\end{lemma}
\begin{proof}
Let $\lambda_1,\cdots,\lambda_n\in L_{(e,c)}$ and $\sum_{i=1}^n k_i\lambda_i\in p(e)\star V_{[c^\infty]}^{q}$ with $k_i\not=0$ for $i=1,...,n$. Since distinct infinite paths are $K$-linearly independent,  $\sum_{i=1}^n k_i\lambda_i\not=0$. Then
\[\sum_{i=1}^n k_i\lambda_i=p(e)\star \sum_{i=1}^m h_j\mu_j=\sum_{j=1}^m h_j(p(e))^{\sigma_{c,\x_q}}\mu_j\]
with $s(e)=s(\mu_j)$, and $h_j\not=0$ for $j=1,...,m$. Without loss of generality we can assume that $t_j:=\max\{t\geq 0: e^t|\mu_j\}\in\mathbb N\cup\{\infty\}$ is an increasing function of $j$, and that, setting $t_0=-1$,  $t_{\bar j}<t_{\bar j +1}=\cdots=t_m$ for a suitable \R{$0\leq \bar j\leq m-1$}. If $e\not=c$, then all $t_j$'s are natural numbers. \\
If the edge $c_1$ does not appear in the cycle $e$, then 
\begin{align*}
\sum_{j=1}^m h_j(p(e))^{\sigma_{c,\x_q}}\mu_j&=\sum_{j=1}^m h_jp(e)\mu_j=\sum_{j=1}^m h_j\left(\sum_{\ell=1}^n p_\ell e^\ell -s(e)\right)\mu_j\\
&=h_mp_ne^n\mu_m+\cdots+h_{t_{\bar j +1}}p_ne^n\mu_{t_{\bar j +1}}+\alpha\\
&=e^{n+t_m}\left (h_mp_n \mu'_m+\cdots+h_{t_{\bar j +1}}p_n\mu'_{t_{\bar j +1}}\right)+\alpha
\end{align*}
where $\mu'_m$, ...,$\mu'_{t_{\bar j +1}}\in L_{c,e}$ and  $\alpha$ is a linear combination of infinite paths each one not left divisible by $e^{n+t_m}$.
Since distinct infinite paths in $[c^\infty]$ are $K$-linearly independent we get that $0=h_m p_n(=\cdots=h_{t_{\bar j +1}})p_n$ and hence $h_m=0$ producing a contradiction.\\
Assume now that $c_1$ is an edge of $e$ and $c\not=e$. 
\begin{align*}
\sum_{j=1}^m h_j(p(e))^{\sigma_{c,\x_q}}\mu_j&=
\sum_{j=1}^m h_j\left(\sum_{\ell=1}^n p_\ell \x_q^\ell e^\ell -s(e)\right)\mu_j\\
&=h_mp_n\x_q^ne^n\mu_m+\cdots+h_{t_{\bar j +1}}p_n\x_q^n e^n\mu_{t_{\bar j +1}}+\alpha\\
&=e^{n+t_m}\left (h_m\x_q^n p_n \mu'_m+\cdots+h_{t_{\bar j +1}}\x_q^n p_n\mu'_{t_{\bar j +1}}\right)+\alpha
\end{align*}
where $\mu'_m$, ...,$\mu'_{t_{\bar j +1}}\in L_{c,e}$ and  $\alpha$ is a $K':=K[x]/\langle q(x)\rangle$-linear combination of infinite paths each one not left divisible by $e^{n+t_m}$.
Since distinct infinite paths in $[c^\infty]$ are $K'$-linearly independent we get that $0=h_mp_n\x_q^n(=\cdots=h_{t_{\bar j +1}}p_n\x_q^n)$; since $\x_q$ is invertible in $K'$ we get $h_m=0$ producing a contradiction.\\
Assume that $c=e$. Then $L_{c,e}=L_{c,c}$ consists of those infinite paths which start at
$s(c)$, and which eventually equal $c^\infty$, but do not start out by traversing $c$. It is non empty if and only if there exists at least a cycle $f\not=c$ with $s(f)=s(c)$. Let $f\not=c$ be a cycle such that $f|\lambda_1$. Multiplying on both sides
\[\sum_{i=1}^n k_i\lambda_i=\sum_{j=1}^m h_j(p(e))^{\sigma_{c,\x_q}}\mu_j\]
by $ff^*$ we get
\[k_1\lambda_1+k_2 ff^*\lambda_2+\cdots+k_n ff^*\lambda_n=\sum_{j=1}^m h_j ff^*\mu_j.\]
If the term on the right is equal to zero, on the left, forgetting the terms equal to 0, we have a $K$-linear combination of distinct infinite paths equal to 0: then their coefficients (in particular $k_1$) are 0 getting a contradiction. Then the set $J_f:=\{j:f|\mu_j\}$ is not empty. The distinct paths $e^n ff^*\mu_j$, $j\in J_f$, are different from any other appearing: therefore since distinct paths are $K'$-linearly independent we get
\[\sum_{j\in J_f} h_j\x_q^ne^n\mu_j=0.\]
Since $\x_q$ is invertible in $K'$, we deduce that $h_j=0$ for $j\in J_f$ contradicting the fact that $J_f\not=\emptyset$.
\end{proof}

\begin{theorem}
If the cycles $c$ and $e$ are not both exclusive, then 
\[\dim_K \Ext^1(V^p_{[c^\infty]}, V^q_{[e^\infty]})=\infty.\]
\end{theorem}
\begin{proof}
Assume now that $c$ is not exclusive and $e\not=c$. Denoted by $f$ a cycle with $s(f)=s(c)$, and by $d$ a path with $s(d)=s(c)$ and $r(d)=r(e)$, we have that $f^id\in L_{c,e}$ for each $i\in\mathbb N$.  If $e$ is not exclusive, denoted by $f$ a cycle with $s(f)=s(e)$, and by $d$ a path with $s(d)=s(c)$ and $r(d)=r(e)$, we have that $df^i\in L_{c,e}$ for each $i\in\mathbb N$. In both the cases we conclude by \Cref{lemma:non excl}.
\end{proof}

\begin{example} Consider the graph
\[\xymatrix@-1pc{&&&s_1\ar@/^/[dr]|{d_1}&&&&t_1\ar@(r,u)|{g'}\ar@/^/[dr]|{g_1}\ar[rr]^h&&w\ar@(r,u)|a\\\
E=&\overline u\ar[r]_p&s_4\ar@/^/[ur]|{d_4}&&s_2\ar@/^/[dl]|{d_2}\ar[rr]^b&&t_3\ar@/^/[ur]|{g_3}&&t_2\ar@<-2pt> `d[l] `[ll]|{g_2} [ll]&&z\ar@(r,u)|\ell \\
&&&s_3\ar@/^/[ul]|{d_3}\ar[rr]_m&&v\ar@/_1pc/ [rrrrru]_n}\]
and the field $K=\mathbb Q$.  The polynomials $p(x)=\frac 12 x^2-1$ and $q(x)=x^3-3x-1$ are basic irreducible in $\mathbb Q[x]$. Consider the cycles $g=g_1g_2g_3$, $d=d_1d_2d_3d_4$, $a$, $\ell$.
The left $L_{\mathbb Q}(E)$-modules $V^p_{[g^\infty]}$, $V^q_{[g^\infty]}$, $V^q_{[d^\infty]}$, $V^p_{[d^\infty]}$, $V^q_{[\ell^\infty]}$, $V^q_{[a^\infty]}$ are simple. 
Then
\begin{enumerate}
\item Since $s(g)\text{Path}(E)s(d)=\emptyset$:
\[\dim_{\mathbb Q}\Ext^1(V^p_{[g^\infty]}, V^q_{[d^\infty]})=0.\]
\item Since $L_{d,\ell}=\{d_1d_2mn\ell^\infty\}$ has cardinality 1, $\deg p(x)=2$, and $\deg q(x)=3$:
\[\dim_{\mathbb Q}\Ext^1(V^p_{[d^\infty]}, V^q_{[\ell^\infty]})=1\times 2\times 3=6.
\]
\item Since $d^\infty$ is exclusive and $p(x)\not=q(x)$:
\[\dim_{\mathbb Q}\Ext^1(V^p_{[d^\infty]}, V^q_{[d^\infty]})=0.
\]
\item Since $d^\infty$ is exclusive and $\deg p(x)=3$:
\[\dim_{\mathbb Q}\Ext^1(V^p_{[d^\infty]}, V^p_{[d^\infty]})=3.
\]
\item Since $L_{d,a}=\{d_1bg_3(g')^{i_1}g^{j_1}\cdots (g')^{i_r}g^{j_r}h a^\infty: 0\leq i_c, j_c, 1\leq c\leq r, r\geq 0\}$ is infinite:
\[\dim_{\mathbb Q}\Ext^1(V^p_{[d^\infty]}, V^q_{[a^\infty]})=\infty.\]
\item Since $g$ is not exclusive:
\[\dim_{\mathbb Q}\Ext^1(V^p_{[d^\infty]}, V^q_{[g^\infty]})=\infty=\dim_{\mathbb Q}\Ext^1(V^p_{[g^\infty]}, V^q_{[a^\infty]}).\]
\end{enumerate}
%\begin{align*}
%\dim_{\mathbb Q}\Ext^1(V^p_{[g^\infty]}, V^q_{[d^\infty]})&=0\\
%\dim_{\mathbb Q}\Ext^1(V^p_{[d^\infty]}, V^q_{[\ell^\infty]})&=1\times 2\times 3=6\\
%\dim_{\mathbb Q}\Ext^1(V^p_{[d^\infty]}, V^q_{[d^\infty]})&=0\\
%\dim_{\mathbb Q}\Ext^1(V^q_{[d^\infty]}, V^q_{[d^\infty]})&=3\\
%\dim_{\mathbb Q}\Ext^1(V^p_{[d^\infty]}, V^q_{[a^\infty]})&=\infty\\
%\dim_{\mathbb Q}\Ext^1(V^p_{[d^\infty]}, V^q_{[g^\infty]})&=\infty\\
%\dim_{\mathbb Q}\Ext^1(V^p_{[g^\infty]}, V^q_{[a^\infty]})&=\infty\\
%\end{align*}
\end{example}

\begin{remark}
The classification of the finitely presented simple modules given in \cite{AR14}, together with the description of the indecomposable injective modules given in \cite{AMT24}, and the results of this paper, 
give a complete description of the finite length modules over a Leavitt Path Algebras with disjoint cycles. we obtain in such a way a more detailed picture of the module category of this class of algebras.
\end{remark}

\end{document}